\documentclass[11pt]{article}
\title{Constructive proofs of some positivstellens\"atze for\\ compact semialgebraic subsets of $\real^d$}
\author{Gennadiy Averkov\footnote{Institute for Mathematical Optimization, Faculty of Mathematics, University of Magdeburg, 39106 Magdeburg, Germany; email: averkov@math.uni-magdeburg.de}}
\usepackage{amsmath,amsthm,amssymb,enumerate}
\usepackage{textcomp}
\usepackage[active]{srcltx}

\newcommand{\natur}{\mathbb{N}}
\newcommand{\real}{\mathbb{R}}
\newcommand{\bF}{\mathbb{F}}
\newcommand{\integer}{\mathbb{Z}}
\newcommand{\setcond}[2]{\left\{ #1 \, : \, #2\right\}}
\newcommand{\biglsetcond}[2]{\bigl\{ #1 \, : \, #2\bigr\}}
\newcommand{\cS}{\mathcal{S}}
\newcommand{\cM}{\mathcal{M}}
\newcommand{\cP}{\mathcal{P}}

\newcommand{\cF}{\mathcal{F}}
\newcommand{\cone}{\operatorname{cone}}
\newcommand{\sign}{\operatorname{sign}}
\newcommand{\term}[1]{\emph{#1}}
\newcommand{\eps}{\varepsilon}
\newcommand{\thmheader}[1]{{\upshape (#1).}}

\newcommand{\overtwocond}[2]{\stackrel{#1}{#2}}

\newtheorem{nn}{}
\newtheorem{theorem}[nn]{Theorem}
\newtheorem{lemma}[nn]{Lemma}

\begin{document}
\maketitle

\begin{abstract}
	In a broad sense, positivstellens\"atze are results about representations of polynomials which are strictly positive on a given set. We give constructive and, to a large extent, elementary proofs of some known positivstellens\"atze for compact semialgebraic subsets of $\real^d$. The presented proofs extend and simplify arguments of Berr, W\"ormann (2001) and Schweighofer (2002, 2005).
\end{abstract}

\newtheoremstyle{itsemicolon}{}{}{\mdseries\rmfamily}{}{\itshape}{:}{ }{}
\newtheoremstyle{itdot}{}{}{\mdseries\rmfamily}{}{\itshape}{:}{ }{}
\theoremstyle{itdot}
\newtheorem*{msc*}{2010 Mathematics Subject Classification} 

\begin{msc*}
	Primary: 14P10;  Secondary: 12Y05, 52B11, 90C30
\end{msc*}


\newtheorem*{keywords*}{Keywords}

\begin{keywords*}
	polyhedron; polytope; positivity; positivstellensatz; preordering; semiring; quadratic module
\end{keywords*}

\section{Introduction}

In what follows $\bF$ is a subfield of $\real$, $d \in \natur$ and $X_1,\ldots,X_d$ are indeterminates. Let $X:=(X_1,\ldots,X_d)$. By $\bF[X]$ denote the ring of all polynomials in indeterminates $X_1,\ldots,X_d$ and with coefficients in $\bF$. A polynomial $f \in \bF[X]$ is called \term{linear} if $f$ has degree at most one. 
For $n \in \natur$ let $[n]:=\{1,\ldots,n\}$ and let $[0]=\emptyset$. If $n \in \natur$, $i \in [n]$ and $u \in \real^d$, then by $u_i$ we denote the $i$-th component of $u$. Given $U \subseteq \real$ let $U_{\ge 0} := \setcond{u \in U}{u \ge 0}$ and $U_{>0} := \setcond{u \in U}{u > 0}$. For $\cF \subseteq \bF[X]$ we define
\[
	\cone_\bF \cF := \setcond{ \sum_{i=1}^n \lambda_i f_i}{n \in \integer_{\ge 0} \ \text{and} \ f_i \in \cF, \lambda_i \in \bF_{\ge 0} \ \forall i \in [n]}.
\]

Throughout the text we consider $a_1,\ldots,a_m \in \bF[X]$ with $m \in \natur$ and $a:=(a_1,\ldots,a_m)$. With $a$ we associate the so-called \term{basic closed set} in $\real^d$ given by 
\[
	\{a_1 \ge 0,\ldots,a_m \ge 0\}:= \setcond{x \in \real^d}{a_1(x) \ge 0,\ldots, a_m(x) \ge 0}.
\]
We study polynomials strictly positive on $\{a_1 \ge 0,\ldots,a_m \ge 0\}$. Results about such polynomials are called positivstellens\"atze. See \cite{MR1659509,MR2383959} for background information from real algebraic geometry and \cite{MR1500280,MR2589247} for various areas of applications. 
By $a$ we also define the following subsets of $\bF[X]$:
\begin{align*}
	\cS_\bF(a) & := \cone_\bF \setcond{a_1^{k_1} \cdots a_m^{k_m}}{k_1,\ldots,k_m \in \integer_{\ge 0}}, \\
	\cP_\bF(a) & := \cone_\bF \setcond{p^2 a_1^{k_1} \cdots a_m^{k_m} }{p \in \bF[X], \ k_1,\ldots,k_m \in \{0,1\}}, \\
	\cM_\bF(a) & :=  \cone_\bF \setcond{p^2 a_i^{k_i}}{p \in \bF[X], \ i \in [m], \ k_i \in \{0,1\}},
\end{align*}
The set $\cS_\bF(a)$ is a semiring, $\cP_\bF(a)$ is a preordering and  $\cM_\bF(a)$ is a quadratic module. We have $\cS_\bF(a) \subseteq \cP_\bF(a)$, $\cM_\bF(a) \subseteq \cP_\bF(a)$ and, if $m=1$, then $\cP_\bF(a)=\cM_\bF(a)$. For the sake of brevity in what follows we shall omit the subscript $\bF$ and write $\cS(a), \cP(a)$ and $\cM(a)$. 

The main aim of this paper is to give a constructive and (mostly) elementary proof of the following theorem. 

\begin{theorem} \label{JP-H-P-S}
Let $S:=\{a_1 \ge 0,\ldots,a_m \ge 0\}$ be nonempty and bounded and let $f \in \bF[X]$ be strictly positive on $S$. Then the following statements hold.
\begin{itemize}
	\item[(JP)] If $\cM(a)$ contains linear polynomials $l_1,\ldots,l_k$, with $k \in \natur$, such that $\{l_1 \ge 0,\ldots,l_k \ge 0\}$ is bounded, then $f \in \cM(a)$.
	\item[(H)] If $a_1,\ldots,a_m$ are all linear, then $f \in \cS(a)$.
	\item[(P)] If for some $g \in \cM(a)$ the set $\{ g \ge 0\}$ is bounded, then $f \in \cM(a)$.
	\item[(S)] One has $f \in \cP(a)$.
\end{itemize}
\end{theorem}

If $a_1,\ldots,a_m$ are all linear and the polyhedron $S=\{a_1 \ge 0,\ldots,a_m \ge 0\}$ is nonempty and bounded, then (JP) implies that every polynomial strictly positive on $S$ necessarily belongs to $\cM(a)$. This was shown for the case $\bF=\real$ by Jacobi and Prestel \cite{MR1817508} with nonconstructive arguments (see also \cite[Theorem~5.3.8, Corollary~6.3.5 and Exercise 6.5.3]{MR1829790}). To the best of author's knowledge no constructive proof of (JP) has previously been available. Assertions~(H), (P) and (S) are well-known theorems of Handelman \cite{MR929582}, Putinar \cite{MR1254128} and Schm\"udgen \cite{MR1092173}, respectively. For further information on Theorem~\ref{JP-H-P-S} see also \cite[Chapters~6,\,7]{MR2383959}. The original proofs of (H), (P) and (S) are highly nonconstructive. Constructive proofs of (H) and (S) were given in \cite{MR1870623} (see also \cite[\S3]{MR1854339} for a related constructive proof of (H)). A constructive proof of (P) for the case $g = \rho - \sum_{i=1}^d X_i^2$, where $\rho \in \bF_{>0}$, was given in \cite{MR2142861}. In this paper we present an elementary and short proof of (H) and show that the arguments from \cite{MR1821179,MR1870623,MR2142861} can be used to give a simple constructive proof of (JP), (P) and (S). Our proof of Theorem~\ref{JP-H-P-S} is elementary with one exception: following \cite{MR1821179,MR1870623} in the proof of (S) we use Stengle's positivstellensatz. Since we prove (P) with the help of (S), also (P) depends on Stengle's positivstellensatz. In contrast to \cite{MR1870623} we do not use Hilbert's basis theorem (see, for example, \cite[Chapter\,2,\S\,5]{MR2290010}). As a consequence, on the algorithmic level one can avoid construction of Gr\"obner bases (see \cite[Chapter\,2]{MR2290010}), which is computationally expensive in general. Below we list the results which are used in the proof of Theorem~\ref{JP-H-P-S}.

\begin{theorem} \thmheader{Affine version of Farkas' lemma \cite[Corollary~7.1h]{MR874114}}
	Let $a_1,\ldots,a_m \in \bF[X]$ be all linear and let the polyhedron $S:=\{a_1 \ge 0,\ldots,a_m \ge 0\}$ be nonempty. Then every linear $f \in \bF[X]$ which is strictly positive on $S$ necessarily belongs to $\cone_\bF \{1,a_1,\ldots,a_m\}$.
\end{theorem}

\begin{theorem}  \thmheader{P\'olya's theorem \cite{Polya1928}, \cite[\S2.24]{MR944909}} \label{polya:thm} Let $f \in \bF[X]$ be homogeneous and strictly positive on the simplex 
\[
	\Delta := \setcond{x \in \real_{\ge 0}^d}{x_1 + \cdots + x_d =1}.
\] 
Then there exists $N \in \integer_{\ge 0}$ such that $(\sum_{i=1}^d X_i)^N f(X) \in \cS(X)$.
\end{theorem}

Note that the proof of Theorem~\ref{polya:thm} given in  \cite{Polya1928} and \cite[\S2.24]{MR944909} is based on elementary arguments. A bound on $N$ can be found in \cite[Theorem~1]{MR1854339}.

\begin{theorem} \thmheader{Stengle's positivstellensatz \cite{MR0332747}} \label{stengle} Let $f \in \bF[X]$ be strictly positive on $S(a)$. Then there exist $g,h \in \cP(a)$ such that $f=(1+g)/(1+h)$.
\end{theorem}

\section{Proofs}

If $l \in \real[X] \setminus \{0\}$ is linear homogeneous and $f \in \real[X] \setminus \{0\}$, we call the polynomial 
$f_0(X) := l(X)^{\deg f} f\left(\frac{X_1}{l(X)},\ldots,\frac{X_d}{l(X)}\right)$
the \term{homogenization} of $f$ with respect to $l$. For $f \in \bF[X]$, writing
$f = \sum_{\alpha} c_\alpha X^\alpha := \sum_{\alpha} c_\alpha X_1^{\alpha_1} \cdots X_d^{\alpha_d}$
we assume that the sum extends over $\alpha \in \integer_{\ge 0}^d$ and the coefficients $c_\alpha \in \bF$ are zero for all but finitely many $\alpha$'s. For $\alpha \in \integer_{\ge 0}^d$ we define $|\alpha|:=\alpha_1 + \cdots + \alpha_d$. We also introduce the notation $\|X\|^2 : = \sum_{i=1}^d X_i^2.$

The following lemma is used in the proof of (H).
\begin{lemma} \label{main lemma}
	Let $f \in \bF[X]$ be strictly positive on $S:=\{a_1 \ge 0,\ldots,a_m \ge 0\}$.
	Let $l_1,\ldots,l_d \in \bF[X]$ be linear and such that $\bF[X]=\bF[l_1,\ldots,l_d]$.
	Let
	$q :=t - \sum_{i=1}^d l_i - \sum_{j=1}^m a_j,$ where $t \in \bF_{>0}$.
	Then $f \in \cS(l_1,\ldots,l_d,a_1,\ldots,a_m,q)$.
\end{lemma}
\begin{proof} Without loss of generality let $(l_1,\ldots,l_d)=X$. We introduce indeterminates $Y_1,\ldots,Y_m$ and $Z$ and define $Y:=(Y_1,\ldots,Y_m)$. Consider 
	\begin{align*}
		\sigma(X,Y,Z) & := \frac{1}{t} \Bigl( \sum_{i=1}^d X_i + \sum_{j=1}^m Y_j + Z \Bigr), \\
		g(X,Y,Z) & := f(X) + C \sum_{j=1}^m (Y_j-a_j(X))^2, \ \text{where} \ C \in \bF_{>0}, \\
		\Delta & :=\setcond{(x,y,z) \in \real_{\ge 0}^d \times \real_{\ge 0}^m \times \real_{\ge 0}}{\sigma(x,y,z) = 1}, \\
		A & := \setcond{(x,y,z) \in \Delta}{y_1=a_1(x), \ldots, y_m =a_m(x)},
	\end{align*}
   For every $C \in \bF_{>0}$ the polynomial $g(X,Y,Z)$
	is strictly positive on $A$. Since $A$ and $\Delta$ are compact, we can fix a sufficiently large $C \in \bF_{>0}$ for which $g$ becomes strictly positive on $\Delta$.
	Let $g_0$ be the homogenization of $g$ with respect to $\sigma$. Then also $g_0$ is strictly positive on $\Delta$. By Theorem~\ref{polya:thm} applied to $g_0$ and the simplex $\Delta$, there exists $N \in \integer_{\ge 0}$ such that  $h(X,Y,Z):=\sigma(X,Y,Z)^N g_0(X,Y,Z) \in \cS(X,Y,Z)$. In $h(X,Y,Z)$ we successively substitute $Z$ with $t - \sum_{i=1}^d X_i - \sum_{j=1}^m Y_j$ and $Y_j$ with $a_j(X)$ for every $j \in[m]$. We obtain $f(X) = h(X,a,q) \in \cS(X,a,q)$.
\end{proof}

\begin{proof}[Proof of (H)]
	Assume that $a_1,\ldots,a_m$ are all linear. We can choose $t_1,\ldots,t_d \in \bF$ such that $l_i:=t_i+X_i$ is nonnegative on $S$ for every $i \in [d]$. Having chosen $t_1,\ldots,t_d$ we choose a sufficiently large $t \in \bF_{>0}$ such that the polynomial $q$ from Lemma~\ref{main lemma} is nonnegative on $S$. By Lemma~\ref{main lemma}, $f \in \cS(l_1,\ldots,l_d, a_1,\ldots,a_m,q)$. By the Farkas lemma $l_1,\ldots,l_d, q\in \cone_\bF \{1,a_1,\ldots,a_m\}$. Hence $ f \in \cS(1,a_1,\ldots,a_m) = \cS(a)$.
\end{proof}

	If $n \in \natur$ and $A_1,\ldots,A_n,B_1,\ldots,B_n$ are indeterminates, then 
	\begin{align} 
		A_1 \cdots A_n \pm B_1 \cdots B_n & = \frac{1}{2^{n-1}} \sum_{e \in E_{\pm}^n} \prod_{i=1}^n (A_i + e_i B_i) \nonumber \\
		& \in \cS(A_1 + B_1,\cdots,A_n+B_n, A_1 - B_1,\ldots,A_n - B_n), \label{prod:id}
	\end{align}
	where $E_+^n$ (resp. $E_-^n$) is the set of all vectors $e \in \{-1,1\}^n$ with even (resp. odd) number of components equal to $-1$. The latter can be easily proved (e.g., by induction on $n$). 

The following lemmas are (essentially) borrowed from \cite{MR1821179,MR1870623,MR2142861}. We somewhat simplify their formulations and the proofs. Lemma~\ref{woermann's first lemma} is a somewhat more explicit version of Lemma~2.1 from \cite{MR1870623} (see also \cite[Lemma~1]{MR1821179}).

\begin{lemma} \label{woermann's first lemma} Let $\rho \in \bF_{>0}$ and let
$f = \sum_\alpha c_\alpha X^\alpha \in \bF[X]$. We define
$t(f,\rho):= \sum_\alpha |c_\alpha| (\rho+1)^{|\alpha|}$.
Then $t(f,\rho) \pm f \in \cP(\rho-\|X\|^2)$.
\end{lemma}
\begin{proof}
	Since $t(f,\rho)=t(-f,\rho)$ it suffices to show $t(f,\rho) + f \in \cP(\rho-\|X\|^2)$. We have
	\[
		t(f,\rho) + f = \sum_\alpha (c_\alpha (\rho+1)^{|\alpha|} \pm c_\alpha X^\alpha) = \sum_{\alpha} |c_\alpha| \left((\rho+1)^{|\alpha|} + \sign(c_\alpha) X^\alpha \right).
	\]
	Let $\alpha$ be an arbitrary multi-index with $\alpha \ne (0,\ldots,0)$. Let us apply \eqref{prod:id} for $n=|\alpha|$. Substituting $A_1, \ldots, A_n$ with $\rho$ and $B_1,\ldots,B_n$ with appropriate  $X_i$'s, we see that $(\rho+1)^{|\alpha|} + \sign(c_\alpha) X^\alpha \in \cP(\rho+1 - X_1,\ldots,\rho+1 - X_d, \rho+1 + X_1,\ldots,\rho+1+X_d)$. For $i \in [d]$  one has
	\begin{equation}
		\rho+1 \pm X_i = \frac{1}{2} \Bigl((\rho+1) + (1 \pm X_i)^2 + \sum_{j \in [d] \setminus \{i\}} X_j^2 + (\rho - \|X\|^2)\Bigr) \in \cP(\rho-\|X\|^2).
			\label{X_i:eq}
	\end{equation}
	Hence $t(f,\rho) + f \in \cP(\rho-\|X\|^2)$.
\end{proof}

Lemma~\ref{domain transfer tool} is similar to Lemma~8 from \cite{MR2142861}.

\begin{lemma} \label{domain transfer tool} Let $S:=\{a_1 \ge 0,\ldots,a_m \ge 0\}$ and let $f \in \bF[X]$ be strictly positive on $S$. Let $B$ be a compact subset of $\real^d$. Then there exists $g \in \cM(a)$ such that $f - g$ is strictly positive on $B$.
\end{lemma}
\begin{proof}
	Let $T:=\setcond{x \in B}{f(x) \le 0}$. We shall use $a$ as the function $x \mapsto (a_1(x),\ldots,a_m(x))$ from $\real^d$ to $\real^m$. The set $a(B)$ is compact. Hence there exists $\gamma \in \bF_{>0}$ such that $a(B) \subseteq (-\infty,2 \gamma]^m$. By the assumption on $f$ we have $a(T) \cap [0,2\gamma]^m = \emptyset$. Since $a(T)$ and $[0,2 \gamma]^m$ are compact, there exists $\eps \in \bF_{>0}$ such that $a(T) \cap [-2\eps,\gamma]^m = \emptyset$. By the choice of $\gamma$ and $\eps$ we see that if $x \in B$ and $a_j(x) \ge -2 \eps$ for every $j \in [m]$, then $f(x) > 0$. Consequently,
	\[
		\mu:= \min \biglsetcond{f(x)}{x \in B \ \text{and} \ a_j(x) \ge - 2 \eps  \ \forall j \in [m]} > 0.
	\]
	Consider the univariate polynomial $h(t):= t \left(\frac{t- \gamma}{\gamma + \eps}\right)^{2N} \in \real[t]$, where $N \in \natur$ is to be fixed below. One has
	\begin{align*}
		 0 \le h(t) \le \gamma \left( \frac{\gamma}{\gamma+\eps}\right)^{2N} =: & c(N) & & \text{on} & &  [0,2 \gamma], \\
		 -h(t) \ge 2 \eps \left( \frac{\gamma+ 2 \eps}{\gamma + \eps} \right)^{2N} =: & C(N) & & \text{on} & & (- \infty, - 2 \eps].
	\end{align*}
	We define
		$g(X) := \sum_{j=1}^m h(a_j(X))$. Let $x \in B$. If $a_j(x) \ge - 2 \eps$ for every $j \in [m]$, we have
		\[
			f(x) - g(x) \ge \mu - \sum_{j=1}^m h(a_j(x)) \ge \mu - \sum_{\overtwocond{j\in [m]}{a_j(x) \ge 0}} h(a_j(x)) \ge \mu - m \, c(N).
		\]
		If $a_j(x) \le - 2 \eps$ for some $j \in [m]$, we have
		\begin{align*}
			f(x)- g(x) \ge \min_{y \in B} f(y) - \sum_{\overtwocond{j \in [m]}{a_j(x) \ge 0}} h(a_j(x)) -  \sum_{\overtwocond{j \in [m]}{a_j(x) < 0}} h(a_j(x)) & \ge \min_{y \in B} f(y) - m \, c(N) + C(N).
		\end{align*}
		Since $c(N) \rightarrow 0$ and $C(N) \rightarrow +\infty$, as $N \rightarrow +\infty$, we deduce $f(x) - g(x)>0$ for every $x \in B$ by choosing $N$ sufficiently large.
\end{proof}

\begin{lemma} \label{schweighofer's second lemma} Let $S:=\{a_1 \ge 0,\ldots,a_m \ge 0\}$ be bounded. Let $\rho \in \bF_{>0}$ and let $\rho-\|X\|^2$ be strictly positive on $S$. Let $f \in \bF[X]$ be strictly positive on $S$. Then $f \in \cM(a,\rho-\|X\|^2)$.
\end{lemma}
\begin{proof}
Fix any linear $l_1,\ldots,l_k \in \bF[X]$ with $k \in \natur$ such that the polyhedron $\{l_1 \ge 0,\ldots,l_k \ge 0\}$ is nonempty and bounded (e.g., one can take $l_1,\ldots,l_k$ with $k=2d$ and $\{l_1 \ge 0,\ldots,l_k \ge 0\} = [0,1]^d$). By Lemma~\ref{woermann's first lemma}, one has $t+ l_1,\ldots,t + l_k \in \cP(\rho-\|X\|^2)$ for every $t \in \bF$ with $t \ge \max_{i \in [k]} t(l_i,\rho)$. The set $B:=\{t+l_1 \ge 0,\ldots,t+l_k \ge 0\}$ is bounded\footnote{This is easy to verify for various concrete choices of $l_1,\ldots,l_k$, e.g., in the case $k=2d$ and $\{l_1 \ge 0,\ldots,l_k \ge 0\}=[0,1]^d$. In the general situation the boundedness of $B$ follows from the fact that $B$ has the same recession cone as $\{l_1 \ge 0,\ldots,l_k \ge 0\}$. See, for example, \cite[\S8.2]{MR874114}.}. By Lemma~\ref{domain transfer tool} there exists $g \in \cM(a)$ such that $f-g$ is strictly positive on $B$. By (H), $f-g \in \cS(t+l_1,\ldots,t+l_k)$. By the choice of $t$ we have $\cS(t+l_1,\ldots,t+l_k) \subseteq \cP(\rho - \|X\|^2) = \cM(\rho - \|X\|^2)$. It follows $f \in \cM(a,\rho-\|X\|^2)$.
\end{proof}

The proof of Lemma~\ref{schweighofer's second lemma} can be compared with the proof of Theorem~3 from \cite[pp.\,8--9]{MR2142861}, in which the author uses P\'olya's theorem rather than (H). Lemma~\ref{generalized woermann's lemma} is a somewhat more general form of Theorem~2.2 from \cite{MR1870623} (see also \cite[The proof of Theorem~4]{MR1821179}).

\begin{lemma} \label{generalized woermann's lemma}
	Let $h \in \bF[X]$ and $\rho \in \bF_{>0}$. Then there exists $\rho' \in \bF_{>0}$ such that $\rho'  - \|X\|^2 \in \cM\bigl(h,(1+h)(\rho-\|X\|^2)\bigr)$.
\end{lemma}
\begin{proof}
	By Lemma~\ref{woermann's first lemma} there exists $t=t(h,\rho)$ such that $t- h \in \cP(\rho-\|X\|^2)$. It follows that
	\begin{align*}
		\cM\bigl((1+h)(\rho-\|X\|^2), h\bigr) & \ni (1 + h)(\rho-\|X\|^2) + h\|X\|^2 + \rho (1 + h)(t- h) + 
\rho(t/2- h)^2 \\ & = \rho (1 + t/2)^2 - \|X\|^2.
	\end{align*}
	Thus, one can define $\rho' := \rho(1+t/2)^2$.
\end{proof}

\begin{proof}[Proof of (JP), (P) and (S)]
	We start with (JP). Assume that $l_1,\ldots,l_k \in \cM(a)$, where $k \in \natur$, are all linear and $\{l_1 \ge 0,\ldots,l_k \ge 0\}$ is bounded. Without loss of generality let $\{ l_1 \ge 0,\ldots,l_k \ge 0\} \subseteq [-1,1]^d$. We notice that 	
	\begin{align*}
		d - \|X\|^2 & = \frac{1}{2}  \sum_{i=1}^d \Bigl((1+X_i)^2 (1-X_i)  + (1-X_i)^2 (1+X_i) \Bigr) \\ & \in \cM(1-X_1,\ldots,1-X_d,1+X_1,\ldots,1+X_d).
	\end{align*}
	By the Farkas lemma $1 \pm X_i \in \cone_\bF(1,l_1,\ldots,l_m) \subseteq \cM(a)$ for every $i \in [d]$. Hence $d - \|X\|^2 \in \cM(a)$. The polynomial $1+d - \|X\|^2$ is strictly positive on $S$ and belongs to $\cM(a)$. Thus, in view of Lemma~\ref{schweighofer's second lemma}, we deduce $f \in \cM(a,1+ d - \|X\|^2) \subseteq \cM(a)$.

	For showing (S) we choose $\rho \in \bF_{>0}$ such that $\rho-\|X\|^2$ is strictly positive on $S$.
By Stengle's positivstellensatz, applied to the polynomial $\rho-\|X\|^2$ strictly positive on $S$, there exist $g, h \in \cP(a)$ such that $\rho-\|X\|^2 = (1+g)/(1+h)$ and $g, h \in \cP(a)$. Hence $(1+h ) (\rho-\|X\|^2) \in \cP(a)$. Then, in view of Lemma~\ref{generalized woermann's lemma}, there exists $\rho' \in \bF_{>0}$ such that $\rho'-\|X\|^2 \in \cP(a)$. By Lemma~\ref{schweighofer's second lemma}, $f \in \cM(a,\rho'-\|X\|^2)$. Thus, $f \in \cP(a)$.

	Let us show (P). Assume $g \in \cM(a)$ and $\{g \ge 0\}$ is bounded. By Lemma~\ref{domain transfer tool} there exists $h \in \cM(a)$ such that $f - h$ is strictly positive on $\{ g \ge 0\}$. By (S), $f-h \in \cP(g) = \cM(g) \subseteq \cM(a)$. Hence $f \in \cM(a)$.
\end{proof}

\providecommand{\bysame}{\leavevmode\hbox to3em{\hrulefill}\thinspace}
\providecommand{\MR}{\relax\ifhmode\unskip\space\fi MR }
\providecommand{\MRhref}[2]{%
  \href{http://www.ams.org/mathscinet-getitem?mr=#1}{#2}
}
\providecommand{\href}[2]{#2}

\end{document}